\documentclass[12pt]{amsart}
\usepackage{txfonts}
\theoremstyle{plain}
\newtheorem{Thm}{Theorem}

\newtheorem{Lem}[Thm]{Lemma}

\renewcommand{\theMain}{}
\newcommand{\rad}{\operatorname{rad}}
\newcommand{\ep}{\epsilon}
\newcommand{\bt}{\bigtriangleup}
\newcommand{\btd}{\bigtriangledown}
\newcommand{\Om}{\Omega}
\begin{document}
\title[Symmetry Results for Monge-Ampere System in the plane]
{Symmetry Results for classical solutions of Monge-Ampere system in
the plane}

\author{Li Ma, Baiyu Liu }

\address{LM: Department of mathematical sciences \\
Tsinghua university \\
Beijing 100084 \\
China} \email{lma@math.tsinghua.edu.cn}

\address{BL: Department of mathematical sciences \\
Tsinghua university \\
Beijing 100084 \\
China} \email{liuby05@mails.tsinghua.edu.cn}

\dedicatory{Dedicated to Louis Nirenberg on the occasion of his 85th
birthday}

\thanks{$^*$ The research is partially supported by the National Natural Science
Foundation of China 10631020 and SRFDP 20060003002. }

\maketitle

\begin{abstract}
In this paper, by the method of moving planes, we prove the symmetry
result which says that classical solutions of Monge-Ampere system in
the whole plane are symmetric about some point. Our system under
consideration comes from the differential geometry problem.


\end{abstract}

\section{Introduction} It is an interesting question to consider the
relation between two convex surfaces in space. There are at least
two famous theorems for convex surfaces in 3-space. One is the
Cohn-Vossen theorem, which says that if two closed convex surfaces
differ by an isometry, then they are the same by a translation. The
other is the Minkowski theorem, which says that if two closed convex
surfaces share the same Gauss curvature, then they are differ only
by a translation. For a simple and beautiful proof of the latter,
one may see the paper of S.S.Chern \cite{chern}.
 We study the relation between two non-compact convex
 surfaces in space, which leads us to
consider a symmetry result for the Monge-Ampere system in the
plane $\mathbf{R}^2$. Assume that we are given two functions
$K_1(x,u,v)$ and $K_2(x,u,v)$ for $(x,u,v)\in \mathbf{R}^2\times
\mathbf{R}^2$. We are looking for a pair of function $u=u(x)$ and
$v=v(x)$ with their graphs $\Gamma(x,u(x))$ and $\Gamma(x,v(x))$
such that $K_1(x,u(x),v(x))$ and $K_2(x,u(x),v(x))$ are Gaussian
curvatures of the graphs $\Gamma(x,u(x))$ and $\Gamma(x,v(x))$
respectively. Then we are lead to solve the Monge-Ampere system
$$
K_1(x,u,v)=\frac{det D^2u(x)}{(1+|Du(x)|^2)^2}, \quad in \quad
\mathbf{R}^2
$$
and
$$
K_2(x,u,v)=\frac{det D^2v(x)}{(1+|Dv(x)|^2)^2}, \quad in \quad
\mathbf{R}^2.
$$
Single Monge-Ampere equation (with Dirichlet boundary condition) has
been studied by many authors, one may see \cite{CafL} and \cite{TU}
for more references.

We study here a little more generalized version of the Monege-Ampere
system above. We study the symmetry result for classical solutions
of the Monege-Ampere system:
\begin{eqnarray}
  \label{eq:sys}
  \left\{
\begin{array}{l@{\quad \quad}l}
det(D^2u)+g(u,v,\nabla u)=0, & in \ \mathbf{R}^{2},\\
det(D^2v)+f(u,v,\nabla v)=0, & in \ \mathbf{R}^{2},\\
\ (D^2u)>0,\ (D^{2}v)>0, & in \ \mathbf{R}^{2},\\
\lim_{\{|x|\to \infty;x_1<\lambda\}}(u(x)-u(x^{\lambda}))\geq 0, & for \ \lambda <0,\\
\lim_{\{|x|\to \infty; x_1<\lambda\}}(v(x)-v(x^{\lambda}))\geq 0, & for \ \lambda <0,\\
\lim_{\{|x|\to \infty;x_1>\lambda\}}(u(x)-u(x^{\lambda}))\geq 0, & for \ \lambda >0,\\
\lim_{\{|x|\to \infty; x_1>\lambda\}}(v(x)-v(x^{\lambda}))\geq 0, & for \ \lambda >0,\\
\lim_{|x|\to \infty}u(x)=\lim_{|x|\to \infty} v(x)=+\infty,\\
\lim_{|x|\to \infty}|\nabla u(x)|=\lim_{|x|\to \infty}|\nabla
v(x)|=\infty,
\end{array}\right.
\end{eqnarray}
for $x^{\lambda}=(2\lambda-x_1,x_2)$ is the reflection of the
point $x$ with respect to any line $T_{\lambda}=\{x_1=\lambda\}$,
where
$f,g \in C^{1}(\mathbf{R}^2\times \mathbf{R}^{2}, \mathbf{R})$.\\
We suppose that:
\begin{eqnarray*}
(i) &   & f\  and\ g\ are\ symmetric in\ p_1: \forall (u,v)\in \mathbf{R}^2, \forall p_1, p_2 \in \mathbf{R},\\
&   & g(u, v, p_{1}, p_2)=g(u, v,-p_{1}, p_2),  f(u, v, p_{1}, p_2)=f(u, v,-p_{1}, p_2);\\
(ii) &   & C_1>\frac{\partial g}{\partial v}(u,v,p_1,p_2)>0, \quad
C_2>\frac{\partial f}{\partial
  u}(u,v,p_1,p_2)>0,\\
&   & \qquad \qquad \qquad \forall (u,v)\in \mathbf{R}^2,
\  \forall p_1, p_2 \in \mathbf{R};\\
(iii)
&   & \lim_{|u|^{-1}+|v|^{-1}+|(p_1, p_2)|^{-1}\to 0 }\frac{\partial g}{\partial u}(u,v,p_1,p_2)<0,\\
&   &
\lim_{|u|^{-1}+|v|^{-1}+|(p_1, p_2)|^{-1}\to 0 }\frac{\partial f}{\partial
  v}(u,v,p_1,p_2)<0;\\
(iv) &   & \lim_{|u|^{-1}+|v|^{-1}+|(p_1, p_2)|^{-1}\to 0 } det \left(
\begin{array}{l@{\quad \quad}l}
 \frac{\partial g}{\partial u}& \frac{\partial g}{\partial v}\\
 \frac{\partial f}{\partial u}& \frac{\partial f}{\partial v}
\end{array}\right)(u,v,p_1,p_2)>0,
\end{eqnarray*}
where $C_i$ are positive constants.

Our result is the following.
\begin{Thm}\label{thm:sym}
Let $(u, v)$ be a classical solution of system (\ref{eq:sys}).
Suppose (i)-(iv) hold, then the solution $(u, v )$ is symmetric in
$x_1$ direction, i.e.,  there is some $\lambda$ such that $u(x_1,
x_2)= u(2\lambda-x_1, x_2)$ and $v(x_1, x_2)= v(2\lambda-x_1,
x_2)$.
\end{Thm}

Notation: We shall use $(\overline{A,B})$ to denote the open
interval in the line from $A$ to $B$.

The method to prove Theorem \ref{thm:sym} is motivated by
\cite{Bu00}, \cite{Cli2}, \cite{ML09}, and \cite{Ma}. We use the
method of moving planes, which has been used in \cite{gnn},
\cite{BN},\cite{ChenL}, and \cite{Cli1}.

\section{Monge-Ampere System in $\mathbf{R}^{2}$}

We prove our theorem in this section.

In what follows, we shall use the method of moving plane. To
proceed, we start by considering lines parallel to $x_1=0$, coming
from $-\infty$. For each $\lambda \in \mathbf{R}$, we define
$$
\Sigma_\lambda := \{x\in \mathbf{R}^2\ | \ x_1 <\lambda\},\quad
T_\lambda :=
\partial \Sigma_\lambda = \{x\in \mathbf{R}^2 \ |\ x_1=\lambda\}.
$$
For any point $x=(x_1, x_2) \in \Sigma_\lambda $, let
$x^\lambda=(2\lambda-x^1, x_2)$ be the reflected point with respect
to the line $T_\lambda$. We define the reflected functions by
$$
u_\lambda(x):= u(x^\lambda), \quad v_\lambda(x):= v(x^\lambda),
$$
and introduce the functions
$$
U_\lambda(x):= u_\lambda(x)-u(x), \quad V_\lambda(x):=
v_\lambda(x)-v(x).
$$

By the definition of $u_{\lambda}$,
$det(D^2u_{\lambda})(x)=det(D^2u)(x^{\lambda})$. By using the
integral form of the theorem of the mean, we obtain:
\begin{equation}
det(D^2 u_\lambda)-det(D^2 u)=a_{ij}U_{\lambda, ij},
\end{equation}
where
\begin{equation}
\label{eq:aij} (a_{ij}(x))=\frac{1}{2}(det(D^2 u_\lambda)(D^2
u_\lambda)^{-1}+det(D^2 u)(D^2 u)^{-1}).
\end{equation}
Noticing that the solution $(u, v)$ are convex everywhere, we have
\begin{equation}
\label{eq:posaij} (a_{ij}(x))>0, \quad x \in \Sigma_{\lambda}.
\end{equation}

On the other hand, by using (i), we have
\begin{eqnarray}
  \label{eq:ulam}
  det(D^2u_{\lambda})(x)& = & -g(u, v, \nabla u)(x^{\lambda})\\
\nonumber & = & -g(u_{\lambda}, v_{\lambda}, -u_{\lambda,1},
  u_{\lambda,2})\\
\nonumber & = & -g(u_{\lambda}, v_{\lambda}, \nabla u_{\lambda}).
\end{eqnarray}
Similarly we have the elliptic equation for $V_{\lambda}$ with the
elliptic coefficient matrix $(b_{ij})=(b_{ij}(x))$.

Therefore, we find that $(U_{\lambda}, V_{\lambda})$ satisfies an
elliptic linear system: in $\Sigma_{\lambda}$
\begin{eqnarray}
  \label{eq:bulam} \\\nonumber
  a_{ij}U_{\lambda,ij}+\frac{\partial g}{\partial p_1}(u_{\lambda}, v_{\lambda}, \theta_1(x,\lambda) ,
  u_{\lambda,2})U_{\lambda,1}
  +\frac{\partial g}{\partial p_2}(u_{\lambda}, v_{\lambda},
  u_{\lambda,1}, \theta_{2}(x,\lambda))U_{\lambda,2} &   &\\
\nonumber  +\frac{\partial g}{\partial u}(\xi_1(x,\lambda), v, \nabla u )U_{\lambda}+\frac{\partial
  g}{\partial v}(u_{\lambda}, \eta_{1}(x, \lambda), \nabla u)V_{\lambda}=0,
\end{eqnarray}
\begin{eqnarray}
  \label{eq:bvlam}\\\nonumber
  b_{ij}V_{\lambda,ij}+\frac{\partial f}{\partial p_1}(u_{\lambda}, v_{\lambda}, \tau_1(x,\lambda) ,
  v_{\lambda,2})V_{\lambda,1}+\frac{\partial f}{\partial p_2}(u_{\lambda}, v_{\lambda},
  v_{\lambda,1}, \tau_{2}(x,\lambda))V_{\lambda,2} &   &\\
\nonumber  +\frac{\partial f}{\partial u}(\xi_2(x,\lambda),
v_\lambda, \nabla v )U_{\lambda}+\frac{\partial
  f}{\partial v}(u, \eta_{2}(x, \lambda), \nabla v)V_{\lambda}=0,
\end{eqnarray}
where for $i=1,2,$
\begin{eqnarray}\label{eq:xieta}
\xi_{i}(x, \lambda) & \in (\overline{u(x),
u_{\lambda}(x)}, \quad
\eta_{i}(x, \lambda) & \in  (\overline {v(x), v_{\lambda}(x)}),
\end{eqnarray}
\begin{equation}
\label{eq:theta} \theta_{i}(x, \lambda)\in (\overline {u_{i}(x),
u_{\lambda,i}(x)}),
\end{equation}
\begin{equation}
  \label{eq:tau}
  \tau_i(x, \lambda)\in (\overline{v_{i}(x), v_{\lambda,i}(x)}).
\end{equation}

We define
$$
\Sigma_\lambda ^{U^-} \coloneqq \{x\in \Sigma_\lambda \ |\ U_\lambda
(x)>0\}, \quad \Sigma_\lambda ^{V^-} \coloneqq \{x\in \Sigma_\lambda
\ |\ V_\lambda (x)>0\}.
$$
For simplicity, we first give two lemmata.

\begin{Lem}
\label{lem:neg}
Assume (i)-(iv) hold. Then there exists a constant
$R_1>0$, such that
$$
    \forall x\in \{\mathbf{R}^2\backslash B_{R_1}(0)\}\cap \Sigma_\lambda
    ^{U^-};\quad
    \forall y\in \{\mathbf{R}^2\backslash B_{R_1}(0)\}\cap \Sigma_\lambda
    ^{V^-}
$$
the following hold
$$
\frac{\partial g}{\partial u}(z, v(x),\nabla u(x))<0, \quad
\frac{\partial f}{\partial v}(u(y), w,\nabla v(y))<0,
$$
where $z\in (\overline{u_\lambda(x), u(x)})$ and $w\in
(\overline{v_\lambda(y), v(y)})$ arbitrarily.
\end{Lem}

\begin{proof}
Using (iii), we choose $\ep>0$ such that $\frac{\partial g}{\partial
u}(u,v,p_1,p_2)<0$, if $|u|^{-1}+|v|^{-1}+|(p_1, p_2)|^{-1}<\ep$. For
this particular $\ep$, since $\lim_{|x|\to
+\infty}(|u(x)|^{-1}+|\nabla u|^{-1})=0$ and $\lim_{|x|\to
+\infty}|v(x)|^{-1}=0$, there exists a constant $R_1>0$, such that
$|u(x)|^{-1}+|v(x)|^{-1}+|\nabla u|^{-1}(x)<\ep$, if $|x|>R_1$. For all
$x\in \{\mathbf{R}^2\backslash B_{R_1}(0)\}\cap \Sigma_\lambda
^{U^-}$, $|u(x)|^{-1}+|v(x)|^{-1}+|\nabla u|^{-1}(x)<\ep$.
Notice that in this case $z\in (\overline{u_\lambda(x) ,u(x)})=(u(x), u_\lambda(x))$. Thus,
$|z|^{-1}+|v(x)|^{-1}+|\nabla u|^{-1}(x)<\ep$, $\forall z\in
(\overline{u_\lambda (x), u(x)})$. Therefore, $\frac{\partial
g}{\partial u}(z, v(x),\nabla u(x))<0$, $\forall z\in
(\overline{u_\lambda (x), u(x)})$. Similarly, when $R_1$ is
sufficiently large, $\frac{\partial f}{\partial v}(u(y), w,\nabla
v(y))<0$, if $ y\in \{\mathbf{R}^2\backslash B_{R_1}(0)\}\cap
\Sigma_\lambda^{V^-}$.
\end{proof}

\begin{Lem}\label{lem:nneg}Assume (i)-(iv) hold. Then
there exists a constant $R_2>0$, such that
$$
\forall \  y_1, y_2 \in \{\mathbf{R}^2 \backslash B_{R_2}(0)\}\cap
\Sigma_\lambda ^{U^-}\cap \Sigma_\lambda ^{V^-},
$$
the following holds:
\begin{eqnarray}\label{eq:det} &   &
\frac{\partial g}{\partial u}(z_1, v(y_1), \nabla
u(y_1))\frac{\partial f}{\partial v}(u(y_2),w_2, \nabla v(y_2)) \\
&   &  \nonumber \quad -\frac{\partial g}{\partial
v}(u_\lambda(y_1),w_1, \nabla u(y_1))\frac{\partial f}{\partial
u}(z_2, v_\lambda (y_2), \nabla v(y_2)) >  0,
\end{eqnarray}
where $z_i \in (\overline{u_\lambda(y_i), u(y_i)})$, $w_i \in
(\overline{v_\lambda(y_i),
v(y_i)})$ $(i=1, 2)$ arbitrarily. \\
\end{Lem}

\begin{proof}
Since $f,g \in C^{1}(\mathbf{R}^2\times \mathbf{R}^{2}, \mathbf{R})$
and the fact that the solutions growing to infinity at infinity, for
all $R\geq 0$, $\forall \ y_1, y_2 \in \{\mathbf{R}^2 \backslash
B_{R}(0)\}\cap \Sigma_\lambda ^{U^-}\cap \Sigma_\lambda ^{V^-}$,
$\frac{\partial g}{\partial u}(z_1, v(y_1), \nabla u(y_1))$ and
$\frac{\partial g}{\partial v}(z_1, v(y_1), \nabla u(y_1))$ are
bounded. Using the continuity of $\frac{\partial f}{\partial u}$ and
$\frac{\partial f}{\partial v}$, we have
\begin{eqnarray*}
&   & \frac{\partial g}{\partial u}(z_1, v(y_1), \nabla
u(y_1))\frac{\partial f}{\partial v}(u(y_2),w_2, \nabla v(y_2)) \\
&   &  \nonumber \quad -\frac{\partial g}{\partial
v}(u_\lambda(y_1),w_1, \nabla u(y_1))\frac{\partial f}{\partial
u}(z_2, v_\lambda (y_2), \nabla v(y_2))\\
&   & \quad \geq \frac{1}{2}\{\frac{\partial g}{\partial u}(z_1,
v(y_1), \nabla
u(y_1))\frac{\partial f}{\partial v}(u(y_1),v(y_1), \nabla v(y_1)) \\
&   &  \nonumber \quad \quad -\frac{\partial g}{\partial v}(u
(y_1),v(y_1), \nabla u(y_1))\frac{\partial f}{\partial u}(u(y_1),
v(y_1), \nabla v(y_1))\}\\
&   & \quad \quad \to \lim_{|u|^{-1}+|v|^{-1}+|(p_1, p_2)|^{-1}\to 0 }\frac{1}{2}(\frac{\partial g}{\partial
u}\frac{\partial f}{\partial v}-\frac{\partial g}{\partial
v}\frac{\partial f}{\partial u})|_{(u,v,p_1,p_2)}>0\quad (as \
R\to +\infty),
\end{eqnarray*}
which proves the lemma.
\end{proof}

We now give our Theorem \ref{thm:sym}, which will be proved by the
use of the maximum principle and the method of moving planes as in
\cite{Bu00}.

\begin{proof}
We just have to show that for $\lambda=0$ there hold
$$
U_\lambda(x)\equiv V_\lambda (x)\equiv 0, \quad \forall x\in
\Sigma_\lambda.
$$
In order to show this, we will apply the moving planes methods in three steps. \\
\textbf{Step 1:} There exists $\lambda^*<0$ such that $U_\lambda
\leq 0$ and $V_\lambda \leq 0$ in $\Sigma_\lambda$, for all $\lambda
\leq \lambda^*$.

Fix $\lambda^*<-\max\{R_1, R_2\}$, then for all $\lambda<\lambda^*$,
we have $\forall x\in \Sigma_\lambda$, $|x|>\max\{R_1, R_2\}$.
Assume for contradiction that there is a $\lambda <\lambda^*$ and a
point $y_0 \in \Sigma_\lambda$, such that $U_\lambda(y_0)>0$.

Since $\lim_{|x|\to +\infty}U_\lambda (x)=0$ and $U_\lambda(x)\equiv
0$, $x\in T_\lambda=\partial \Sigma_\lambda$, we may take $y_1\in
\Sigma_\lambda$, such that
$$
U_\lambda(y_1)=\max_{y\in \bar{\Sigma}_\lambda}U_\lambda(y)>0.
$$

At point $y_1$, we have $(U_{\lambda, ij})(y_1)\leq 0$ and $\nabla
U_\lambda(y_1)=0$. Thus, (\ref{eq:bulam}) turns to be
\begin{equation}
\label{eq:ulam}
 \frac{\partial g}{\partial u}(\xi_1(y_1,\lambda), v(y_1),
\nabla u (y_1))U_{\lambda}(y_1)+\frac{\partial
  g}{\partial v}(u_{\lambda}(y_1), \eta_{1}(y_1, \lambda), \nabla u(y_1))V_{\lambda}(y_1)\geq
  0.
\end{equation}
Notice that $y_1 \in \{\mathbf{R}^2\backslash B_{R_1}(0)\}\cap
\Sigma_{\lambda}^{U^-}$, by using Lemma \ref{lem:neg}, we have
$$
\frac{\partial
  g}{\partial v}(u_{\lambda}(y_1), \eta_{1}(y_1, \lambda), \nabla
  u(y_1))V_{\lambda}(y_1)> 0.
$$
Combining this with the assumption (ii), we have $V_\lambda
(y_1)>0$, i.e. $y_1 \in \{\mathbf{R}^2 \backslash B_{R_2}(0)\}\cap
\Sigma_\lambda ^{U^-} \cap \Sigma_\lambda ^{V^-}$. On the other
hand, we also have
$$
 \lim_{|x|\to
+\infty}V_\lambda (x)=0 \quad \textrm{and}\quad  V_\lambda(x)=0\quad
\forall \ x \in T_\lambda,
$$
hence, we can take $y_2 \in \Sigma_\lambda$ such that
$$
V_\lambda (y_2)=\max_{y\in \bar{\Sigma}_\lambda }V_\lambda (y)>0.
$$
We can repeat the above argument for (\ref{eq:bvlam}) and show that
$U_\lambda (y_2)>0$ i.e. $y_2 \in \{\mathbf{R}^2 \backslash
B_{R_2}(0)\}\cap \Sigma_\lambda ^{U^-} \cap \Sigma_\lambda ^{V^-}$
and
\begin{equation}
\label{eq:vlam} \frac{\partial f}{\partial u}(\xi_2(y_2,\lambda),
v_\lambda(y_2), \nabla v (y_2))U_{\lambda}(y_2)+\frac{\partial
  f}{\partial v}(u(y_2), \eta_{2}(y_2, \lambda), \nabla v(y_2))V_{\lambda}(y_2)\geq
  0.
\end{equation}
Let us put
$$
J_{11}(\lambda)=\frac{\partial g}{\partial u}(\xi_1(y_1,\lambda),
v(y_1), \nabla u (y_1))<0, $$
$$
J_{12}(\lambda)=\frac{\partial g}{\partial v}(u_{\lambda}(y_1),
\eta_{1}(y_1, \lambda), \nabla u(y_1))> 0,
$$
$$
J_{21}(\lambda)=\frac{\partial f}{\partial u}(\xi_2(y_2,\lambda),
v_\lambda(y_2), \nabla v (y_2))> 0,$$
$$
J_{22}(\lambda)=\frac{\partial
  f}{\partial v}(u(y_2), \eta_{2}(y_2, \lambda), \nabla v(y_2))<0.
$$
By using (\ref{eq:ulam}) and (\ref{eq:vlam}), we obtain
\begin{eqnarray*}
U_\lambda(y_1) & \leq &
-\frac{J_{12}(\lambda)}{J_{11}(\lambda)}V_\lambda(y_1)\\
& \leq & -\frac{J_{12}(\lambda)}{J_{11}(\lambda)}V_\lambda(y_2)\\
& \leq & \frac{J_{12}(\lambda)J_{21}(\lambda)}{J_{11}(\lambda)J_{22}(\lambda)}U_\lambda(y_2)\\
& \leq & \frac{J_{12}(\lambda)J_{21}(\lambda)}{J_{11}(\lambda)J_{22}(\lambda)}U_\lambda(y_1),
\end{eqnarray*}
which implies
$$
J_{11}(\lambda)J_{22}(\lambda)-J_{12}(\lambda)J_{21}(\lambda)\leq 0.
$$
However, as we have shown $y_1, y_2 \in \{\mathbf{R}^2 \backslash
B_{R_2}(0)\}\cap \Sigma_\lambda ^{U^-} \cap \Sigma_\lambda ^{V^-}$,
by the result of Lemma \ref{lem:nneg},
$$
J_{11}(\lambda)J_{22}(\lambda)-J_{12}(\lambda)J_{21}(\lambda)>0,
$$
which leads to a contradiction. \textbf{Step 1 is completed.}\\

We now move the plane $T_\lambda$ toward right, i.e. increase the
value of $\lambda$, as long as both $U_\lambda (x)\leq 0$ and
$V_\lambda (x)\leq 0$ hold. Define:
\begin{equation}\label{def:lamb}
\lambda_0=\sup \{\lambda \in \mathbf{R} \ | \ \forall \mu \leq
\lambda, \ U_\mu \leq 0, V_\mu \leq 0, \forall x \in \Sigma_{\mu}
\}.
\end{equation}
Step 1 implies that $\lambda_0 >-\infty$. On the other hand, as
$u(x)>u(0)$, $\forall |x|>R$, for some sufficiently large $R$,
$\lambda_0<+\infty$. We want to show that $\lambda_0\geq 0$.
Assume not, i.e., $\lambda_0<0$.

Since all objects we consider are continuous with respect to
$\lambda$, we know that
$$U_{\lambda_0}\leq 0\quad  \textrm{and} \quad  V_{\lambda_0} \leq
0,\quad \textrm{in} \ \Sigma_{\lambda_0}.
$$
Then it follows from (\ref{eq:bulam}) (\ref{eq:bvlam}) and (ii) that
in $\Sigma_{\lambda_0}$:
\begin{eqnarray}
  \label{eq:u} \\\nonumber
  a_{ij}U_{\lambda_0,ij} &+&\frac{\partial g}{\partial p_1}(u_{\lambda_0}, v_{\lambda_0}, \theta_1(x,\lambda_0) ,
  u_{\lambda_0,2})U_{\lambda_0,1}\\
\nonumber  &+& \frac{\partial g}{\partial p_2}(u_{\lambda_0},
v_{\lambda_0},
  u_{\lambda_0,1}, \theta_{2}(x,\lambda_0))U_{\lambda_0,2} +\frac{\partial g}{\partial u}(\xi_1(x,\lambda_0), v,
\nabla u )U_{\lambda_0}\geq 0,
\end{eqnarray}
\begin{eqnarray}
  \label{eq:v}\\\nonumber
  b_{ij}V_{\lambda_0,ij}&+&\frac{\partial f}{\partial p_1}(u_{\lambda_0}, v_{\lambda_0}, \tau_1(x,\lambda_0) ,
  v_{\lambda_0,2})V_{\lambda_0,1}\\
\nonumber  &+&\frac{\partial f}{\partial p_2}(u_{\lambda_0},
v_{\lambda_0},
  v_{\lambda_0,1}, \tau_{2}(x,\lambda_0))V_{\lambda_0,2} +\frac{\partial
  f}{\partial v}(u, \eta_{2}(x, \lambda_0), \nabla v)V_{\lambda_0}\geq 0,
\end{eqnarray}
where $\theta_i, \tau_i, \eta_i$ are the same as in (\ref{eq:bulam})
and (\ref{eq:bvlam}).\\
\textbf{Step 2.} Either $U_{\lambda_0}\equiv 0$ or
$V_{\lambda_0}\equiv 0$ hold in $\Sigma_{\lambda_0}$.

Suppose $U_{\lambda_0}\nequiv 0$ and $V_{\lambda_0}\nequiv  0$, in
$\Sigma_{\lambda_0}$. Thus we can choose $R_3>\max\{R_1, R_2\}$ s.t.
$U_{\lambda_0}\nequiv 0$ and $V_{\lambda_0}\nequiv 0$ in
$\Sigma_{\lambda_0}\cap B_{R_3}(0)$.

The definition of $\lambda_0$ implies that there exist sequences
$\{\lambda_k\}_{k=1}^\infty \subset \mathbf{R}$ and
$\{y_k\}_{k=1}^\infty \subset \mathbf{R}^2$ such that
$\lambda_k>\lambda_0$, $\lim_{k\to \infty}\lambda_k=\lambda_0$,
$y_k\in \Sigma_{\lambda_k}$, and either $U_{\lambda_k}(y_k)>0$, or
$V_{\lambda_k}(y_k)>0$. Taking $U_{\lambda_k}(y_k)>0$ (up to a
subsequence) as an example, we can rename $y_k$ to be the maximum
points, i.e.
\begin{equation}\label{def:yk}
U_{\lambda_k}(y_k)=\max_{x\in \Sigma_{\lambda_k}}U_{\lambda_k}(x)>0,
\quad k =1,2,\dots.
\end{equation}
Then there are two cases could happen.\\
\textbf{Case 1.} The sequence $\{y_k\}_{k=1}^\infty$ contains a
bounded subsequence.

Without loss of generality, we assume $\{y_k\}_{k=1}^\infty\subset
\{\Sigma_{\lambda_0}\cap B_{R_3-1}(0)\}.$
$$
\lim_{k\to \infty}y_k=y_0, \quad  y_0\in
\cap_{k=1}^{+\infty}\Sigma_{\lambda_k}=\overline{\Sigma_{\lambda_0}}.
$$
Noticing that since $(a_{ij})$ is  strictly convex in
$\Sigma_{\lambda_0}\cap B_{R_3}(0)$, we can apply the Maximum
Principle and the Hopf's lemma to (\ref{eq:u}) in
$\{\Sigma_{\lambda_0}\cap B_{R_3}(0)\}$.

From (\ref{def:yk}), we have $U_{\lambda_k}(y_k)>0$, $ D^2 U_{\lambda_k}(y_k)\geq 0$, and
$\nabla U_{\lambda_k}(y_k)=0$, which imply
\begin{equation}\label{eq:uy0}
U_{\lambda_0}(y_0)\geq 0,\quad
D^2 U_{\lambda_0}(y_0)\geq 0, \quad \textrm{and} \quad \nabla
U_{\lambda_0}(y_0)=0,
\end{equation}
Since $U_{\lambda_0}(x)\leq 0$ $\forall x\in \Sigma_{\lambda_0}$, we have $U(y_0)=0$. Using the Maximum
Principle, we have either (a) $U_{\lambda_0}\equiv 0$ in
$\{\Sigma_{\lambda_0}\cap B_{R_3}(0)\}$, or (b) $U_{\lambda_0}< 0$
$\forall x\in \{\Sigma_{\lambda_0}\cap B_{R_3}(0)\}$.  Since we
haven chosen $R_3$ to be a large constant such that (a) does not
happen, (b) mush hold. Therefore, $y_0\in \partial
\{\Sigma_{\lambda_0}\cap B_{R_3}(0)\}$. By using Hopf's lemma, we
have $\frac{\partial U_{\lambda_0}}{\partial \nu}(y_0)>0$, ($\nu$ is
the outward normal vector on $\partial \{\Sigma_{\lambda_0}\cap
B_{R_3}(0)\}$), which contradicts the
second equation in (\ref{eq:uy0}).\\
\textbf{Case 2.} $\lim_{k\to \infty}|y_k|=+\infty$.

In this case, we can choose $k^*>0$ such that if $k>k^*$,
$|y_k|>\max\{R_1, R_2\}$, where $R_1, R_2$ as in Lemma \ref{lem:neg}
and Lemma \ref{lem:nneg}. For any fixed $k>k^*$, exactly as in Step
1, we can show $V_{\lambda_k}(y_k)>0$. After choosing
$$
V_{\lambda_k}(x_k)=\max_{y\in \Sigma_{\lambda_k}}V_{\lambda_k}(y)>0,
$$
we also have $U_{\lambda_k}(x_k)>0$. Therefore,
$$
x_k, y_k \in \{\mathbf{R}^2 \backslash B_{R_2}(0)\}\cap
\Sigma_{\lambda_k} ^{U^-}\cap \Sigma_{\lambda_k} ^{V^-}.
$$
Similarly defining
$$
J_{11}({\lambda_k})=\frac{\partial g}{\partial
u}(\xi_1(y_k,\lambda_k), v(y_k), \nabla u (y_k))<0,
$$
$$
J_{12}(\lambda_k)=\frac{\partial g}{\partial v}(u_{\lambda_k}(y_k),
\eta_{1}(y_k, \lambda_k), \nabla u(y_k))>0,
$$
$$
J_{21}(\lambda_k)=\frac{\partial f}{\partial
u}(\xi_2(x_k,\lambda_k), v_\lambda(x_k), \nabla v (x_k))>0,
$$
$$
J_{22}(\lambda_k)=\frac{\partial
  f}{\partial v}(u(x_k), \eta_{2}(x_k, \lambda_k), \nabla v(x_k))<0,
$$
same as in Step 1, we can show
$J_{11}(\lambda_k)J_{22}(\lambda_k)-J_{12}(\lambda_k)J_{21}(\lambda_k)\leq
0,$ which contradicts Lemma \ref{lem:nneg}.
\textbf{Step 2 is completed.}\\
\textbf{Step 3.} Both $U_{\lambda_0}\equiv 0$ and
$V_{\lambda_0}\equiv 0$ hold in $\Sigma_{\lambda_0}$.

From Step 2, we let $U_{\lambda_0}\equiv 0$  for example, i.e.
$u(x)$ is symmetry about $T_{\lambda_0}$. Putting $U_{\lambda_0}$
into (\ref{eq:bulam}) and using the assumption (ii), we obtain
$V_{\lambda_0}\equiv 0$. Assume that $\lambda_0\geq 0$. Doing the
same for planes from the right, we know that there is some
$\lambda_1$ such that both $U_{\lambda_1}\equiv 0$ and
$V_{\lambda_1}\equiv 0$ hold in
$\Sigma_{\lambda_1}^c$.\textbf{Step 3 is completed.}

\end{proof}

\end{document}